\documentclass[11pt]{amsart}
\usepackage[dvipsnames]{xcolor}
\usepackage[utf8]{inputenc}
\usepackage{amssymb,amsmath,amsthm,amsfonts,fullpage,float,latexsym,bbm,microtype,cite,fancyvrb}
\usepackage{tikz}
\usepackage{enumerate}
\usetikzlibrary{decorations.pathreplacing}
\usepackage{hyperref}
\hypersetup{colorlinks=true, linkcolor=blue, citecolor=magenta, filecolor=magenta, urlcolor=magenta}
\usepackage{bm}

\makeatletter
\newtheorem*{rep@theorem}{\rep@title}\newcommand{\newreptheorem}[2]{%
\newenvironment{rep#1}[1]{%
\def\rep@title{\bf #2 \ref{##1}}%
\begin{rep@theorem}}%
{\end{rep@theorem}}}
\makeatother
\newreptheorem{theorem}{Theorem}

\newtheorem{theorem}{Theorem}[section]
\newtheorem{proposition}[theorem]{Proposition}

\newtheorem{conjecture}[theorem]{Conjecture}
\newtheorem{lemma}[theorem]{Lemma}
\newtheorem{corollary}[theorem]{Corollary}
\theoremstyle{definition}
\newtheorem{remark}[theorem]{Remark}

\newcommand{\Q}{\mathbb{Q}}

\newcommand{\Z}{\mathbb{Z}}
\newcommand{\C}{\mathbb{C}}
\newcommand{\N}{\mathbb{N}}

\DeclareMathOperator{\Hilb}{Hilb}
\DeclareMathOperator{\Frob}{Frob}

\DeclareMathOperator{\Stir}{Stir}

\DeclareMathOperator{\Char}{Char}
\DeclareMathOperator{\GL}{GL}
\DeclareMathOperator{\IrrChar}{IrrChar}

\DeclareMathOperator{\Sym}{Sym}

\DeclareMathOperator{\Res}{Res}
\DeclareMathOperator{\ch}{char}

\DeclareRobustCommand{\qbinom}{\genfrac[]{0pt}{}}

\begin{document}

\title{Diagonal Supersymmetry for Coinvariant Rings}

\author[J. Lentfer]{John Lentfer}
\address{Department of Mathematics\\
         University of California, Berkeley, CA, USA}
\email{jlentfer@berkeley.edu}


\begin{abstract}
For finite groups $G$, we show that bosonic-fermionic coinvariant rings have a natural $U(\mathfrak{gl}(k|j)) \otimes \mathbb{C}[G]$-module structure.
In particular, we show that their character series are sums of super Schur functions $s_\lambda(\mathbf{q}/\mathbf{u})$ times irreducible characters of $G$ with universal coefficients, which do not depend on $k,j$. 
In the case where $G$ is the symmetric group with diagonal action, this proves the ``Diagonal Supersymmetry'' conjecture of F. Bergeron (2020).
\end{abstract}

\maketitle

\section{Introduction}\label{sec:introduction}

The diagonal coinvariant ring $R_n^{(2,0)} = \C[\bm{x},\bm{y}]/\allowbreak \langle \C[\bm{x},\bm{y}]_+^{\mathfrak{S}_n} \rangle$ was introduced by Haiman in 1994 \cite{Haiman1994}, and has played a central role in contemporary algebraic combinatorics (see for example \cite{CarlssonMellit2018,HaglundRemmelWilson2018,Mellit,DAdderioMellit,BHMPS-Delta,BHMPS-Paths}).
Here $\bm{x}$ and $\bm{y}$ denote two sets of variables $\{x_1,\ldots,x_n\}$ and $\{y_1,\ldots,y_n\}$, respectively. The defining ideal $\langle \C[\bm{x},\bm{y}]_+^{\mathfrak{S}_n} \rangle$ is generated by all polynomials in $\C[\bm{x},\bm{y}]$ without constant term, which are invariant under the diagonal action of the symmetric group $\mathfrak{S}_n$. Haiman conjectured \cite{Haiman1994} and then later proved \cite{Haiman2002} formulas for the dimension, bigraded Hilbert series, and Frobenius series of $R_n^{(2,0)}$ using novel results on the Hilbert scheme of points in $\C^2$. 

Work by F. Bergeron on coinvariant rings $R_n^{(k,0)}$ with $k$ sets of commuting variables \cite{Bergeron2013, BergeronOPAC} revealed that, as the number of sets of variables grows, their structure stabilizes. 
As $k$ grows large enough, their decomposition as $\GL(k) \times \mathfrak{S}_n$-modules is eventually governed by coefficients that do not depend on $k$ \cite{Bergeron2013}. 
Nonetheless, working out the explicit module structure, or even the dimension, remains difficult when $k \geq 3$, and a proof of either for $R_n^{(3,0)}$ has been elusive \cite{Haiman1994, BergeronPrevilleRatelle}.

In a second line of work, anticommuting variables were brought into the coinvariant ring story. 
The coinvariant ring $R_n^{(1,1)}$, with one set of commuting variables and one set of anticommuting variables, was initially studied by N. Bergeron, Colmenarejo, Li, Machacek, Sulzgruber, and Zabrocki in a working seminar in algebraic combinatorics at the Fields Institute in 2018. 
Subsequently, the multiplicity of its sign character \cite{SwansonWallach1}, its dimension and Hilbert series \cite{RhoadesWilson2023}, a monomial basis \cite{SaganSwanson2024, Angarone2024}, and its Frobenius series \cite{MuraiRhoadesWilson} were all determined. 
Zabrocki introduced the coinvariant ring $R_n^{(2,1)}$, with two sets of commuting variables and one set of anticommuting variables \cite{Zabrocki2019}, conjecturally related to a symmetric function appearing in the Delta conjecture of Haglund, Remmel, and Wilson \cite{HaglundRemmelWilson2018}.\footnote{One version of the Delta conjecture has been proven by \cite{DAdderioMellit} and \cite{BHMPS-Delta}, but the connection between that and Zabrocki's coinvariant ring $R_n^{(2,1)}$ remains conjectural.} 

These are all instances of coinvariant rings $R_n^{(k,j)}$ with $k$ sets of $n$ commuting (bosonic, even) variables $\bm{x}^{(1)}, \bm{x}^{(2)}, \ldots, \bm{x}^{(k)}$, where $\bm{x}^{(i)} = \{x_1^{(i)},\ldots,x_n^{(i)}\}$, and $j$ sets of $n$ anticommuting (fermionic, odd) variables $\bm{\theta}^{(1)}, \bm{\theta}^{(2)}, \ldots, \bm{\theta}^{(j)}$, where $\bm{\theta}^{(i)} = \{\theta_1^{(i)},\ldots,\theta_n^{(i)}\}$, for $k,j$ nonnegative integers (see equation~\eqref{eq:coinv} for the complete definition). 
F. Bergeron and others have proposed to study coinvariant rings in this general bosonic-fermionic setting (see \cite{Bergeron2020, BHIR, Zabrocki2020}). 
In particular, these two lines of work are connected by F. Bergeron's ``Diagonal Supersymmetry'' conjecture \cite{BIRS, Bergeron2020}, which predicts that the purely bosonic case already determines the full bosonic-fermionic case. 
We prove this conjecture in Corollary~\ref{cor:main-theorem}.

We first describe the construction of the polynomial ring $\C[\bm{x}^{(1)}, \ldots, \bm{x}^{(k)}, \bm{\theta}^{(1)}, \ldots, \bm{\theta}^{(j)}]$.
Let $G \subset \GL(V) = \GL(n)$ be a finite group acting on $V = \C^n$. 
Let $\C^{k|j}$ denote the complex superspace of even dimension $k$ and odd dimension $j$. 
Then $\Sym(\C^{k|j} \otimes V) = (\Sym V)^{\otimes k} \otimes (\bigwedge V)^{\otimes j}$, and upon choosing a basis of $V$, we label the variables so that they correspond to the tensor factors. 
That is, $\bm{x}^{(1)},\ldots, \bm{x}^{(k)}$ correspond to variables for the $k$ copies of the symmetric algebra $\Sym V$ and $\bm{\theta}^{(1)}, \ldots, \bm{\theta}^{(j)}$ correspond to variables for the $j$ copies of the exterior algebra $\bigwedge V$.
Thus we write $\Sym(\C^{k|j} \otimes V) = \C[\bm{x}^{(1)}, \ldots, \bm{x}^{(k)}, \bm{\theta}^{(1)}, \ldots, \bm{\theta}^{(j)}]$, where 
commuting variables commute with all variables;
anticommuting variables anticommute with all anticommuting variables. That is, $\theta_i^{(\ell)} \theta_h^{(m)} = - \theta_h^{(m)} \theta_i^{(\ell)}$, which implies that $(\theta_i^{(\ell)})^2 = 0$. 

We now define the \textit{$(k,j)$-bosonic-fermionic coinvariant ring for a finite group $G \subset \GL(n)$} by
\begin{equation}\label{eq:coinv} R_G^{(k,j)} = \C[\bm{x}^{(1)}, \ldots, \bm{x}^{(k)}, \bm{\theta}^{(1)}, \ldots, \bm{\theta}^{(j)}]/\langle \C[\bm{x}^{(1)}, \ldots, \bm{x}^{(k)}, \bm{\theta}^{(1)}, \ldots, \bm{\theta}^{(j)}]_+^{G} \rangle,\end{equation}
where $\C[\bm{x}^{(1)}, \ldots, \bm{x}^{(k)}, \bm{\theta}^{(1)}, \ldots, \bm{\theta}^{(j)}]^G_+$ denotes the polynomials without constant term, which are invariant under the diagonal action of $G$ on each set of variables.
In the case where $G$ is the symmetric group $\mathfrak{S}_n$ acting diagonally by permuting the indices of the variables within each set, we denote this by $R_n^{(k,j)}$.

In this paper, we propose that a natural setting for studying $R_G^{(k,j)}$ is as a $U(\mathfrak{gl}(k|j)) \otimes \C[G]$-module, where $U(\mathfrak{gl}(k|j))$ is the universal enveloping algebra of the Lie superalgebra $\mathfrak{gl}(k|j)$. 
Recall that a $G$-representation $G \to \GL(V)$ is equivalent to a $\C[G]$-module \cite[\S21.1]{BourbakiAlgebraIIX}.
In general, given any Lie superalgebra $\mathfrak{g}$, a $\mathfrak{g}$-module is equivalent to a $U(\mathfrak{g})$-module, where $U(\mathfrak{g})$ is the universal enveloping algebra of $\mathfrak{g}$ \cite[Chapter 1.5.1]{ChengWang}.
We will frequently utilize these equivalences.

We describe the relevant actions on $\Sym(\C^{k|j} \otimes V)$.
First, $\mathfrak{gl}(k|j)$ acts on $\C^{k|j}$ as the defining representation. 
Since $\mathfrak{gl}(k|j)$ acts trivially on $V=\C^n$, then $\mathfrak{gl}(k|j)$ acts on $\C^{k|j} \otimes V$.
As $\C^{k|j} \otimes V$ is the degree 1 (in total degree) part of $\Sym(\C^{k|j} \otimes V)$, the action of $\mathfrak{gl}(k|j)$ extends to the whole ring $\Sym(\C^{k|j} \otimes V)$ by left superderivations (see equation~\eqref{eq:superderivations}).
Next, since $G \subset \GL(V)$, then $G$ acts in a natural way on $V$, and trivially on $\C^{k|j}$. 
Combining these, $G$ acts on $\C^{k|j} \otimes V$, which extends to $\Sym(\C^{k|j} \otimes V)$.
The actions of $\mathfrak{gl}(k|j)$ and $G$ commute with each other (cf. \cite[Section 5.2.1]{ChengWangBook}).

With this action in hand, we then show that $R_G^{(k,j)}$ is a quotient module of $\Sym(\C^{k|j} \otimes V)$.
Since $\mathfrak{gl}(k|j)$ is not semisimple, we are not a priori guaranteed that $R_G^{(k,j)}$ will decompose into simple $U(\mathfrak{gl}(k|j))$-modules. 
Our first main result is that the desired $U(\mathfrak{gl}(k|j)) \otimes \C[G]$-module structure exists for $R_G^{(k,j)}$, and we describe how it decomposes into simple modules.
Denote the set of partitions $\lambda$ with length $\ell(\lambda) \leq n$ which satisfy $\lambda_{k+1} \leq j$ by $P(k,j,n)$.
By $P(\infty,j,n)$, $P(k,\infty,n)$, or $P(\infty,\infty,n)$, we mean the set of partitions $\lambda$ with length $\ell(\lambda) \leq n$.
Let $\mu$ index the irreducible characters of $G$.
Let $s_\lambda(\mathbf{q}/\mathbf{u})$ denote a super Schur function (see Section~\ref{sec:background} for definitions), where $\mathbf{q}$ denotes $(q_1,\ldots,q_k)$, and $\mathbf{u}$ denotes $(u_1,\ldots,u_j)$. 
\begin{theorem}\label{thm:new-main-first-part}
Fix a positive integer $n$ and a finite group $G \subset \GL(n)$. 
There is an isomorphism of $U(\mathfrak{gl}(k|j)) \otimes \C[G]$-modules:
    \begin{equation}
        R_G^{(k,j)} \cong \bigoplus_{\lambda \in P(k,j,n)} \bigoplus_{\mu} \left(U^\lambda_{k|j} \otimes N^\mu\right)^{\oplus c_{\lambda\mu}},
    \end{equation}
    where the $U^\lambda_{k|j}$ are simple $U(\mathfrak{gl}(k|j))$-modules with character $s_\lambda(\mathbf{q}/\mathbf{u})$ and the $N^\mu$ are simple $\C[G]$-modules, for some nonnegative integer coefficients $c_{\lambda\mu}$.
\end{theorem}

Let $\IrrChar(G)$ denote the set of irreducible characters of $G$.
From the module structure in Theorem~\ref{thm:new-main-first-part}, taking $\mathfrak{gl}(k|j)$ and $G$-characters yields a character series in terms of super Schur functions and $G$-characters. 
A priori, for fixed $G \subset \GL(V)$, these coefficients are functions of $(k,j)$. However, in our second main result we show that they are independent of $(k,j)$.
\begin{theorem}\label{thm:G-main-theorem}
    Fix a positive integer $n$ and a finite group $G \subset \GL(n)$. 
    For partitions $\lambda$ with $\ell(\lambda) \leq n$, and indices $\mu$ of irreducible $G$-characters, there exist nonnegative integer coefficients $c_{\lambda\mu}$ such that for any $(k,j)$, the multigraded character series of $R_G^{(k,j)}$ is
\begin{equation}\label{eq:G-theorem}
    \Char(R_G^{(k,j)}; \mathbf{q};\mathbf{u}) =  \sum_{\lambda \in P(k,j,n)} \sum_{\chi^\mu \in \IrrChar(G)}c_{\lambda\mu} s_\lambda(\mathbf{q}/\mathbf{u})\chi^\mu.
\end{equation}
\end{theorem}

When $G$ is the symmetric group acting by permuting variables, we may apply the Frobenius characteristic map to the symmetric group characters, which are indexed by partitions of $n$. 
Let $\mathbf{z}$ be an (infinite) auxiliary set of variables for the Frobenius characters.
Then the following result is an immediate consequence of Theorem~\ref{thm:G-main-theorem}. Thus we prove a conjecture of F. Bergeron.

\begin{corollary}[{Diagonal Supersymmetry \cite[Conjecture 1]{Bergeron2020}}]\label{cor:main-theorem}
    Fix a positive integer $n$. 
    For partitions $\lambda$ with $\ell(\lambda) \leq n$, and $\mu \vdash n$, there exist nonnegative integer coefficients $c_{\lambda\mu}$ such that for any $(k,j)$, the multigraded Frobenius series of $R_n^{(k,j)}$ is
\begin{equation}\label{eq:theorem}
    \Frob(R_n^{(k,j)}; \mathbf{q};\mathbf{u}) = \sum_{\lambda \in P(k,j,n)} \sum_{\mu \vdash n} c_{\lambda\mu} s_\lambda(\mathbf{q}/\mathbf{u})s_\mu(\mathbf{z}).
\end{equation}
\end{corollary}
Equation~\eqref{eq:theorem} says that the mixed bosonic-fermionic Frobenius series is obtained from the purely bosonic Frobenius series by replacing each $s_\lambda(\mathbf{q})$ with $s_\lambda(\mathbf{q}/\mathbf{u})$; this is the sense in which the bosonic case determines the mixed case. 
We make this determination explicit in Corollaries~\ref{cor:diagonal-supersymmetry-threshold} and~\ref{cor:infinite-threshold}.

\begin{remark}\label{rmk:too-small}
Given fixed $n,k,j$, it is important to note that equation~\eqref{eq:theorem} does not determine $c_{\lambda \mu}$ for all partitions $\lambda$ of length at most $n$, but only for those which satisfy $\lambda_{k+1} \leq j$. 
\end{remark}

Previously, F. Bergeron showed \cite[equation (2.1)]{Bergeron2020} that $R_n^{(k,j)}$ is a $\GL(k) \times \GL(j) \times \mathfrak{S}_n$-module.
The present approach strengthens this by exploiting the Lie superalgebra structure that leverages the relationship between the even and odd parts.
In this case, the irreducible characters of polynomial representations of $\GL(k)$ are Schur functions $s_\lambda$ where $\ell(\lambda) \leq k$ (see for example \cite[I.A.8]{Macdonald}). 
Taking $\GL(k) \times \GL(j)$ characters and Frobenius characters, one obtains a series which is a product of Schur functions in three sets of variables. That is, for partitions $\lambda, \nu$, and for $\mu \vdash n$, there exist nonnegative integers $c_{\lambda\nu\mu}$ such that for any $(k,j)$,
\begin{equation}\label{eq:a-priori}
    \Frob(R_n^{(k,j)}; \mathbf{q};\mathbf{u}) = \sum_{\substack{\lambda,\nu,\\ \ell(\lambda) \leq k, \ell(\nu) \leq j}}\sum_{\mu \vdash n}  c_{\lambda\nu\mu} s_\lambda(\mathbf{q}) s_\nu(\mathbf{u})s_\mu(\mathbf{z}).
\end{equation}  
Since a super Schur function $s_\lambda(\mathbf{q}/\mathbf{u})$ can always be written as a sum of products of Schur functions $s_\rho(\mathbf{q})s_\nu(\mathbf{u})$, but the converse does not hold, equation~\eqref{eq:theorem} implies equation~\eqref{eq:a-priori}, but equation~\eqref{eq:a-priori} does not imply equation~\eqref{eq:theorem}.

Define $R_n^{(\infty,\infty)}$ as the direct limit of the modules $R_n^{(k,j)}$ over the directed set $\N \times \N$ ordered componentwise, with transition maps $R_n^{(k,j)}\to R_n^{(k',j')}$ for $(k,j)\leq(k',j')$ induced by the inclusions of polynomial rings $\C[\bm{x}^{(1)},\dots,\bm{x}^{(k)},\bm{\theta}^{(1)},\dots,\bm{\theta}^{(j)}]\hookrightarrow\C[\bm{x}^{(1)},\dots,\bm{x}^{(k')},\bm{\theta}^{(1)},\dots,\bm{\theta}^{(j')}]$ that send each variable to itself.
These descend to graded homomorphisms on the quotients because any $G$-invariant in the smaller ring remains $G$-invariant in the larger one. 
They are moreover injective by Lemma~\ref{lem:intersection}.

The limit $R_n^{(\infty,\infty)}$ has an $\N^\infty \times \N^\infty$ multigrading, where $\N^\infty$ denotes the set of nonnegative integer sequences with finitely many nonzero terms.
Let $\Lambda(\mathbf{z})$ denote the ring of symmetric functions in the variables $\mathbf{z}$.
For any $\N^\infty \times \N^\infty$ multigraded $\mathfrak{S}_n$-module $W$ with finite-dimensional homogeneous components, one can define a multigraded Frobenius series
\begin{equation}
    \Frob(W; \mathbf{q}; \mathbf{u}) = \sum_{(\alpha, \beta) \in \N^\infty \times \N^\infty} \mathbf{q}^\alpha \mathbf{u}^\beta \Frob(W_{(\alpha, \beta)}) \in \Z[[\mathbf{q},\mathbf{u}]] \otimes \Lambda(\mathbf{z}).
\end{equation}
Here $\mathbf{q} = (q_1,q_2,\ldots)$ and $\mathbf{u} = (u_1,u_2,\ldots)$ are infinite formal alphabets, and $\mathbf{q}^\alpha \mathbf{u}^\beta$ records the multidegree of the component $W_{(\alpha, \beta)}$.
This formula recovers the standard definition of the multigraded Frobenius series in the finite case.

If $W = \varinjlim W^{(k,j)}$ for some finite-dimensional, $\N^{k} \times \N^{j}$ multigraded modules $\{W^{(k,j)}\}$ with graded transition maps, then
\begin{equation}
    \Frob(W ;\mathbf{q};\mathbf{u}) = \varprojlim_{k,j} \Frob(W^{(k,j)} ; q_1,\ldots,q_{k};u_1,\ldots,u_{j}),
\end{equation}
where the inverse limit is taken with respect to the projection maps obtained by setting $q_{k+1},q_{k+2},\dots$ and $u_{j+1},u_{j+2},\dots$ to zero.
In particular, for $R_n^{(k,j)}$, the index set $P(k,j,n)$ stabilizes once $k  \geq n$ to the set of all partitions $\lambda$ such that $\ell(\lambda) \leq n$.

F. Bergeron showed \cite{Bergeron2020} that there is a well-defined universal series
\begin{equation}\label{eq:francois-universal-series}
    \Frob(R_n^{(\infty,\infty)}; \mathbf{q};\mathbf{u}) =  \sum_{\lambda,\nu} \sum_{\mu \vdash n}c_{\lambda\nu\mu} s_\lambda(\mathbf{q})s_\nu(\mathbf{u})s_\mu(\mathbf{z}),
\end{equation}
which determines all $c_{\lambda\nu\mu}$ which could appear in the expansion in equation~\eqref{eq:a-priori} for any $(k,j)$.

We strengthen equation~\eqref{eq:francois-universal-series} by showing that there is a well-defined universal series in terms of super Schur functions. 
\begin{proposition}\label{prop:universal-series}
Fix a positive integer $n$.
    There is a well-defined \textit{universal series}
\begin{equation}\label{eq:universal-series}
    \Frob(R_n^{(\infty,\infty)}; \mathbf{q};\mathbf{u}) =  \sum_{\lambda,\ \ell(\lambda) \leq n} \sum_{\mu \vdash n}c_{\lambda\mu} s_\lambda(\mathbf{q}/\mathbf{u})s_\mu(\mathbf{z}),
\end{equation}
which determines all $c_{\lambda\mu}$ which could appear in the expansion in equation~\eqref{eq:theorem} for any $(k,j)$.
\end{proposition}

In general, $c_{\lambda\mu}$ may be zero or nonzero.
For some finite $k$ or $j$ in the expansion of $\Frob(R_n^{(k,j)}; \mathbf{q};\mathbf{u})$ (equation~\eqref{eq:theorem}), not all nonzero $c_{\lambda\mu}$ from the universal series (equation~\eqref{eq:universal-series}) need to appear, since they are multiplied by $s_\lambda(\mathbf{q} / \mathbf{u})$, which equals $0$ when $k$ or $j$ is too small.

This raises the question of when there are enough nonzero coefficients $c_{\lambda\mu}$ determined by the multigraded Frobenius series of $R_n^{(k,j)}$ to determine the multigraded Frobenius series of $R_n^{(k',j')}$. 
By Proposition~\ref{prop:universal-series}, the multigraded Frobenius series of $R_n^{(k,j)}$ determines the multigraded Frobenius series of $R_n^{(k',j')}$ whenever $k' \leq k$ and $j' \leq j$. 
In Proposition~\ref{prop:negative}, we show that for sufficiently large $n$, neither of the graded Frobenius series of $R_n^{(1,0)}$ and $R_n^{(0,1)}$ determines enough nonzero $c_{\lambda\mu}$ to determine the other; similarly none of the bigraded Frobenius series of $R_n^{(2,0)}$, $R_n^{(1,1)}$, and $R_n^{(0,2)}$ determine enough nonzero $c_{\lambda\mu}$ to determine another.

The outline of the paper is as follows.
After reviewing the background in Section~\ref{sec:background}, we prove Theorems \ref{thm:new-main-first-part} and \ref{thm:G-main-theorem} in Section~\ref{sec:bergeron-conjecture}.
In Section~\ref{sec:discussion}, we study the universal series for the symmetric group, where we prove Propositions~\ref{prop:universal-series} and~\ref{prop:negative}. 
In Section~\ref{sec:hilbert}, we conclude by discussing further combinatorial consequences of cancellation on the Hilbert and Frobenius series of certain $R_n^{(k,j)}$. 

\section{Background}\label{sec:background}

We assume basic familiarity with symmetric functions (see \cite[Chapter 7]{StanleyEC2} or \cite[Chapter I]{Macdonald}), and representation theory of finite groups and Lie algebras (see \cite{BourbakiAlgebraIIX}).
Assume all tensor products are taken over $\C$.

For a nonnegative integer $d$, the $q$-integer is $[d]_q = 1 + q + q^2 + \cdots +q^{d-1}$ and the $q$-factorial is $[d]_q! = [d]_q[d-1]_q \cdots [2]_q[1]_q$, where we take $[0]_q! = 1$.
For integers $m \geq d \geq 0$, the $q$-binomial coefficient is
\begin{equation}
    \qbinom{m}{d}_q = \frac{[m]_q!}{[d]_q!\,[m-d]_q!},
\end{equation}
which is also the generating function $\sum_\lambda q^{|\lambda|}$ over all partitions $\lambda$ inside the rectangular shape $(m-d)^d$.
By convention, set $\qbinom{m}{0}_q=1$ for all integers $m$ and set $\qbinom{m}{d}_q = 0$ if $d < 0$ or $d > m \geq 0$.

We write $[m]_{u,v} = u^{m-1} + u^{m-2}v + \cdots + uv^{m-2} + v^{m-1} = s_{(m-1)}(u,v)$ for a two-variable deformation of $m$.

If a ring $R$ decomposes as a direct sum of multihomogeneous $\C[G]$-modules, for $G$ a finite group:
\begin{equation}\label{eq:R-grading}
R = \bigoplus_{r_1,\ldots,r_k, s_1, \ldots, s_j \geq 0} (R)_{r_1,\ldots,r_k, s_1, \ldots, s_j},
\end{equation}
then its \textit{multigraded character series} is
\begin{equation} 
\Char(R; \mathbf{q}; \mathbf{u}) := \sum_{r_1,\ldots,r_k, s_1, \ldots, s_j \geq 0} \ch\left((R)_{r_1,\ldots,r_k, s_1, \ldots, s_j} \right)q_1^{r_1} \cdots q_k^{r_k}u_1^{s_1} \cdots u_j^{s_j},
\end{equation}
where $\ch(S)$ denotes the character of a $G$-module $S$.

Recall that the Frobenius characteristic map $F$ is an isomorphism between the group of virtual $\mathfrak{S}_n$-characters and the degree-$n$ symmetric functions $\Lambda_\Z^n$, given by $F(\chi^\mu) = s_\mu(\mathbf{z})$, where $\chi^\mu$ is the irreducible $\mathfrak{S}_n$-character indexed by $\mu$, and $s_\mu(\mathbf{z})$ is a Schur function (see \cite[Chapter 7.18]{StanleyEC2}). 
In the special case where $G$ is the symmetric group $\mathfrak{S}_n$, such as when $R = R_n^{(k,j)}$, we can apply the Frobenius characteristic map $F$ to the character of each multigraded component, and obtain the \textit{multigraded Frobenius series:} 
\begin{equation} \Frob(R_n^{(k,j)}; \mathbf{q}; \mathbf{u}) := \sum_{r_1,\ldots,r_k, s_1, \ldots, s_j \geq 0} F\ch\left((R_n^{(k,j)})_{r_1,\ldots,r_k, s_1, \ldots, s_j} \right)q_1^{r_1} \cdots q_k^{r_k}u_1^{s_1} \cdots u_j^{s_j}.\end{equation}
We will often use $q,t$ for $q_1,q_2$ and $u,v$ for $u_1,u_2$ when $k,j \leq 2$.

$\Sym(\C^{k|j} \otimes V)$ is multigraded with respect to $\mathbf{q}$ and $\mathbf{u}$, where each variable in $\bm{x}^{(i)}$ has $q_i$-degree of $1$ and each variable in $\bm{\theta}^{(i)}$ has $u_i$-degree of $1$.
Furthermore, $\Sym(\C^{k|j} \otimes V)$ is graded by total degree, where each variable has degree $1$.
That is,
\begin{equation}\label{eq:total-degree}
    \Sym(\C^{k|j} \otimes V) = \bigoplus_{d \geq 0}\Sym^d(\C^{k|j} \otimes V),
\end{equation}
where $\Sym^d(\C^{k|j} \otimes V)$ is the $d$th symmetric power.

Define a \textit{super Schur function\footnote{Also called a hook Schur function---not to be confused with a Schur function of a hook shaped partition.}} (see \cite{BereleRegev} or \cite[Appendix A.2.2]{ChengWangBook}) by 
\begin{equation}\label{eq:superschur}
    s_\lambda(\mathbf{q}/\mathbf{u}) = \sum_{\nu \subseteq \lambda} s_\nu(\mathbf{q}) s_{\lambda'/\nu'}(\mathbf{u}).
\end{equation}
Super Schur functions can also be defined in terms of super tableaux (see \cite[Chapter 12.3]{Musson}).
While we will not otherwise use it, we note that in \cite{Bergeron2020} the plethystic notation $s_\lambda [\mathbf{q} - \varepsilon \mathbf{u}]$ is used. Since (see for example \cite[Section 2]{Bergeron2020})
\begin{equation} s_\lambda [\mathbf{q} - \varepsilon \mathbf{u}] = \sum_{\nu \subseteq \lambda} s_\nu(\mathbf{q}) s_{\lambda'/\nu'}(\mathbf{u}),
\end{equation}
then $s_\lambda(\mathbf{q}/\mathbf{u}) = s_\lambda [\mathbf{q} - \varepsilon \mathbf{u}]$. 
We will exclusively use $s_\lambda(\mathbf{q}/\mathbf{u})$.

Super Schur functions satisfy the following key properties \cite[I.3 Example 23]{Macdonald}:
\begin{itemize}
    \item Cancellation:
    \begin{equation}\label{eq:cancellation}
        s_\lambda(q_1,\ldots,q_{k-1},q_k/u_1,\ldots,u_{j-1},u_j)|_{u_j = -q_k} = s_\lambda(q_1,\ldots,q_{k-1}/u_1,\ldots,u_{j-1}),
    \end{equation}
    \item Restriction:
    \begin{equation}\label{eq:restriction1}
        s_\lambda(q_1,\ldots,q_{k-1},q_k/u_1,\ldots,u_j)|_{q_k = 0} = s_\lambda(q_1,\ldots,q_{k-1}/u_1,\ldots,u_{j}),
    \end{equation}
    \begin{equation}\label{eq:restriction2}
        s_\lambda(q_1,\ldots,q_k/u_1,\ldots,u_{j-1},u_j)|_{u_j = 0} = s_\lambda(q_1,\ldots,q_k/u_1,\ldots,u_{j-1}).
    \end{equation}
\end{itemize}
Also, for fixed $(k,j)$ and each $m \geq 0$, the super Schur functions
$\{s_\lambda(\mathbf{q}/\mathbf{u}) \, | \, \lambda \vdash m,\ \lambda_{k+1} \leq j\}$ are linearly independent \cite[Lemma 6.4]{BereleRegev}.

The super Schur functions $s_\lambda(\mathbf{q}/\mathbf{u})$ are characters of the simple $\mathfrak{gl}(k|j)$-modules $U^\lambda_{k|j}$ which appear in the decomposition of $\Sym(\mathbb{C}^{k|j} \otimes V)$ via Howe duality (see \cite{Howe1989, Howe}).
The following result is due to Howe \cite{Howe1989} (see also \cite[Theorem 5.19; Chapter 5: Notes]{ChengWangBook}).

\begin{theorem}[$(\mathfrak{gl}(k|j), \GL(n))$-Howe Duality]\label{thm:super-Howe}
    For all $d \geq 0$, we have that
    \begin{equation}
        \Sym^d(\C^{k|j} \otimes V) \cong \bigoplus_{\substack{\lambda \in P(k,j,n)\\ \lambda \vdash d}} U^\lambda_{k|j} \otimes U^\lambda_{n},
    \end{equation}
    where for each $\lambda \in P(k,j,n)$, $U^\lambda_{k|j}$ is a simple $\mathfrak{gl}(k|j)$-module with character the super Schur function $s_\lambda(\mathbf{q}/\mathbf{u})$, and $U^\lambda_{n}$ is a simple $\GL(n)$-module with character the Schur function $s_\lambda$.
\end{theorem}

We note that taking characters of both sides of $(\mathfrak{gl}(k|j), \GL(n))$-Howe Duality gives a super Cauchy identity, which can also be used to compute the super Schur functions (see \cite[Chapter 5.2.2]{ChengWangBook}).

We describe the action of $\mathfrak{gl}(k|j)$ by superderivations.
For $1 \leq a,b \leq k$ and $1 \leq c,d \leq j$, we have
\begin{equation}\label{eq:superderivations}
    E_{a,b}^+ = \sum_{p=1}^n x_p^{(a)} \partial_{x_p^{(b)}},\
    E_{a,d}^\pm = \sum_{p=1}^n x_p^{(a)} \partial_{\theta_p^{(d)}},\
    E_{c,b}^\mp = \sum_{p=1}^n \theta_p^{(c)} \partial_{x_p^{(b)}},\
    E_{c,d}^- = \sum_{p=1}^n \theta_p^{(c)} \partial_{\theta_p^{(d)}},
\end{equation}
which realize the action of $\mathfrak{gl}(k|j)$ on $\Sym(\C^{k|j} \otimes V)$ \cite[Lemma 5.13]{ChengWangBook}. 
At a glance, the notation means that the $E^+$ operators raise and lower the degree of certain bosonic variables, the $E^\pm$ operators raise the degree of certain bosonic variables and lower the degree of certain fermionic variables, the $E^\mp$ operators raise the degree of certain fermionic variables and lower the degree of certain bosonic variables, and the $E^-$ operators raise and lower the degree of certain fermionic variables.
In the special case of $(k,j) = (1,1)$, see also \cite[Section 3]{SwansonWallach1} and \cite[Section 4.1]{SwansonWallach2}.
Note that the action of $\mathfrak{gl}(k|j)$ by superderivations preserves each degree $d$ component (with respect to the total degree) $\Sym^d(\C^{k|j} \otimes V)$.

We also have that $\GL(V)$ acts by matrix multiplication on $V$ and trivially on $\C^{k|j}$.
This action also preserves $\Sym^d(\C^{k|j} \otimes V)$.
Since $G \subset \GL(V)$, $G$ acts by matrix multiplication on $V$ and trivially on $\C^{k|j}$, and the action preserves $\Sym^d(\C^{k|j} \otimes V)$.

We conclude this section with a consequence of the Jordan-H\"older theorem, on the multiplicities of simple modules occurring in a composition series of a quotient module (see \cite[\S11.3]{BourbakiAlgebraIIX}).

\begin{lemma}\label{lem:composition-series}
    Let $A$ be an associative algebra over $\C$. 
    Let $M$ be a finite-dimensional $A$-module, and let $N$ be an $A$-submodule of $M$. 
    Then for any simple $A$-module $S$ in a composition series of $M/N$, its multiplicity in $M/N$ is less than or equal to its multiplicity in $M$.
\end{lemma}

\section{Proofs of the main results}\label{sec:bergeron-conjecture}

We prove the first main theorem from the introduction.
\begin{proof}[Proof of Theorem~\ref{thm:new-main-first-part}] 
We continue with the setup from the introduction, where $V = \C^n$ and
    \begin{align}
    \begin{aligned}
        \C[\bm{x}^{(1)}, \ldots, \bm{x}^{(k)}, \bm{\theta}^{(1)}, \ldots, \bm{\theta}^{(j)}] &= (\Sym V)^{\otimes k} \otimes (\bigwedge V)^{\otimes j}\\
        &= \Sym(V\otimes \C^{k|0} \oplus V\otimes\C^{0|j})\\
        &= \Sym(\C^{k|j} \otimes V).
    \end{aligned}
    \end{align}
    Our strategy is to first study $\Sym(\C^{k|j} \otimes V)$ and then consider the quotient $R_G^{(k,j)}$. 
    We will work with $\Sym^d(\C^{k|j} \otimes V)$, the part of homogeneous total degree $d$.
    
    Consider the simple $\GL(n)$-modules $U^\lambda_{n}$ which occur in Howe duality (Theorem~\ref{thm:super-Howe}).
    Since $V = \C^n$, $\GL(V)$ is $\GL(n)$, and we identify their modules.
    From $\GL(V)$ we may restrict to any finite subgroup $G$ of $\GL(V)$. Upon restricting the representations $U^\lambda_n$, since any finite-dimensional $G$-representation is semisimple, we have that
    \begin{equation}\label{eq:restriction}
        \Res^{\GL(V)}_G(U_n^\lambda) = \bigoplus_{\mu} (N^\mu)^{\oplus d_{\lambda\mu}},
    \end{equation}
    where the sum is over an indexing set for the simple $G$-representations $N^\mu$, and the $d_{\lambda\mu}$ are some nonnegative integers.
    In particular, the $d_{\lambda\mu}$ are finite because $U^\lambda_n$ is finite-dimensional.
    
    Since the actions preserve $\Sym^d(\C^{k|j} \otimes V)$, it is a $(\mathfrak{gl}(k|j),G)$-module (equivalently, a $U(\mathfrak{gl}(k|j)) \otimes \mathbb{C}[G]$-module). 
    By Howe duality (Theorem~\ref{thm:super-Howe}) and equation~\eqref{eq:restriction}, $\Sym^d(\C^{k|j} \otimes V)$ has decomposition:
    \begin{align}\label{eq:sym-decomposition}
    \begin{aligned}
        \Sym^d(\C^{k|j} \otimes V)  &\cong \bigoplus_{\substack{\lambda \in P(k,j,n)\\\lambda\vdash d}} U^\lambda_{k|j} \otimes \left(\bigoplus_{\mu} (N^\mu)^{\oplus d_{\lambda\mu}}\right)\\
        &\cong \bigoplus_{\substack{\lambda \in P(k,j,n)\\\lambda\vdash d}} \bigoplus_{\mu}U^\lambda_{k|j} \otimes \left(N^\mu\right)^{\oplus d_{\lambda\mu}}.
    \end{aligned}
    \end{align}

    Let $\Sym(\C^{k|j} \otimes V)^G_d$ denote the total degree $d$ polynomials in $\Sym(\C^{k|j} \otimes V)^G$.
    Since the $\mathfrak{gl}(k|j)$-action is degree-preserving and the $\mathfrak{gl}(k|j)$ and $G$-actions commute, $\Sym(\C^{k|j} \otimes V)^G_d$ is a $U(\mathfrak{gl}(k|j)) \otimes \C[G]$-submodule of $\Sym(\C^{k|j} \otimes V)$.
    Then 
    \begin{equation}
        \Sym(\C^{k|j} \otimes V)^G_+ = \bigoplus_{d \geq 1} \Sym(\C^{k|j} \otimes V)^G_d
    \end{equation}
    is a $U(\mathfrak{gl}(k|j)) \otimes \mathbb{C}[G]$-submodule of $\Sym(\C^{k|j} \otimes V)$.

    We turn our attention to $R_G^{(k,j)}$. From its definition, we have that
    \begin{equation}
        R_G^{(k,j)} = \Sym(\C^{k|j} \otimes V)/ \langle \Sym(\C^{k|j} \otimes V)^{G}_+ \rangle.
    \end{equation}
    Consider the multiplication map $m: \Sym(\C^{k|j} \otimes V) \otimes \Sym(\C^{k|j} \otimes V) \to \Sym(\C^{k|j} \otimes V)$. 
    Restrict to the map $m': \Sym(\C^{k|j} \otimes V) \otimes \Sym(\C^{k|j} \otimes V)^G_+ \to \Sym(\C^{k|j} \otimes V)$.
    The image of $m'$, which is $\langle \Sym(\C^{k|j} \otimes V)^{G}_+ \rangle$, is a $ U(\mathfrak{gl}(k|j)) \otimes \mathbb{C}[G]$-submodule of $\Sym(\C^{k|j} \otimes V)$.
    Thus $R_G^{(k,j)}$ is a quotient module of $\Sym(\C^{k|j} \otimes V)$.

    Since $U^\lambda_{k|j}$ is a finite-dimensional simple $U(\mathfrak{gl}(k|j))$-module and $N^\mu$ is a finite-dimensional simple $\C[G]$-module, then $U^\lambda_{k|j} \otimes N^\mu$ is a finite-dimensional simple $U(\mathfrak{gl}(k|j)) \otimes \C[G]$-module, by \cite[\S12.1 Theorem 1(a)]{BourbakiAlgebraIIX}.
    This implies that 
    \begin{equation}\label{eq:sym-d-expansion}
        \Sym^d(\C^{k|j} \otimes V)  
        \cong \bigoplus_{\substack{\lambda \in P(k,j,n)\\\lambda\vdash d}} \bigoplus_{\mu} \left(U^\lambda_{k|j} \otimes N^\mu\right)^{\oplus d_{\lambda\mu}}
    \end{equation}
    is a decomposition into simple $U(\mathfrak{gl}(k|j)) \otimes \C[G]$-modules, and is semisimple.
    Then $(R^{(k,j)}_G)_d$, the homogeneous part of total degree $d$ of $R^{(k,j)}_G$, is semisimple, by \cite[\S4.1 Corollary 3]{BourbakiAlgebraIIX}. 
    In the decomposition in equation~\eqref{eq:sym-d-expansion}, there are finitely many $\lambda$ indexing the sum and the decomposition in equation~\eqref{eq:restriction} is into finitely many simple modules.
    Thus $\Sym^d(\C^{k|j} \otimes V)$ has a direct sum decomposition into finitely many simple modules, so it is of finite length.
    Then by Lemma~\ref{lem:composition-series},
    \begin{equation}
        (R^{(k,j)}_G)_d  
        \cong \bigoplus_{\substack{\lambda \in P(k,j,n)\\\lambda\vdash d}} \bigoplus_{\mu} \left(U^\lambda_{k|j} \otimes N^\mu\right)^{\oplus c_{\lambda\mu}},
    \end{equation}
    for some integers $0 \leq c_{\lambda\mu} \leq d_{\lambda\mu}$.
    By summing over all $d$, we get that
    \begin{equation}\label{eq:last-in-proof}
        R_G^{(k,j)} \cong \bigoplus_{\lambda \in P(k,j,n)} \bigoplus_{\mu} \left(U^\lambda_{k|j} \otimes N^\mu\right)^{\oplus c_{\lambda\mu}}.
    \end{equation}
\end{proof}

Next, we prove a lemma, which will be used in the subsequent proof of Theorem~\ref{thm:G-main-theorem}.
\begin{lemma}\label{lem:intersection}
    For $k'\leq k$ and $j'\leq j$, regard $\Sym(\C^{k'|j'}\otimes V)\subseteq \Sym(\C^{k|j}\otimes V)$ as the subalgebra generated by the first $k'$ bosonic and first $j'$ fermionic variable sets. Then
    \begin{equation}\label{eq:intersection}
        \langle \Sym(\C^{k|j}\otimes V)^G_+ \rangle \cap \Sym(\C^{k'|j'}\otimes V) = \langle \Sym(\C^{k'|j'}\otimes V)^G_+ \rangle,
    \end{equation}
    where the left ideal is taken in $\Sym(\C^{k|j}\otimes V)$ and the right in $\Sym(\C^{k'|j'}\otimes V)$.
\end{lemma}

\begin{proof}
    The general identity follows by removing one variable set at a time and intersecting, so it suffices to treat a single bosonic step (the fermionic step is analogous):
    \begin{equation}\label{eq:intersection-step}
        \langle \Sym(\C^{k|j}\otimes V)^G_+ \rangle \cap \Sym(\C^{k-1|j}\otimes V) = \langle \Sym(\C^{k-1|j}\otimes V)^G_+ \rangle.
    \end{equation}
    The containment $\supseteq$ is immediate, since $\Sym(\C^{k-1|j}\otimes V)^G_+\subseteq \Sym(\C^{k|j}\otimes V)^G_+$ and the right-hand ideal lies in $\Sym(\C^{k-1|j}\otimes V)$.

    For the reverse containment $\subseteq$, let $\pi\colon \Sym(\C^{k|j}\otimes V) \to \Sym(\C^{k-1|j}\otimes V)$ be the $\C$-algebra homomorphism setting $\bm{x}^{(k)}=0$. Since $G$ acts diagonally on each set of variables, it preserves $q_k$-degree, so $\pi$ is $G$-equivariant. Take $f \in \langle \Sym(\C^{k|j}\otimes V)^G_+ \rangle \cap \Sym(\C^{k-1|j}\otimes V)$ and write $f=\sum_i a_i g_i$ with $a_i\in \Sym(\C^{k|j}\otimes V)$ and $g_i\in \Sym(\C^{k|j}\otimes V)^G_+$. As $f$ has zero $q_k$-degree, $f = \pi(f) = \sum_i \pi(a_i)\pi(g_i)$. Each $\pi(g_i)$ is the zero $q_k$-degree part of $g_i$, hence $G$-invariant with the same vanishing constant term, so $\pi(g_i)\in \Sym(\C^{k-1|j}\otimes V)^G_+$. Therefore $f\in \langle \Sym(\C^{k-1|j}\otimes V)^G_+ \rangle$.
\end{proof}

Now we prove our second main theorem.
\begin{proof}[Proof of Theorem~\ref{thm:G-main-theorem}]
    Given the decomposition in Theorem~\ref{thm:new-main-first-part}, take the $U(\mathfrak{gl}(k|j)) \otimes \C[G]$ character to conclude that
        \begin{equation}\label{eq:k-j-proof}
        \Char(R_G^{(k,j)};\mathbf{q};\mathbf{u}) = \sum_{\lambda \in P(k,j,n)} \sum_{\chi^\mu \in \IrrChar(G)} c_{\lambda\mu} s_\lambda(\mathbf{q}/\mathbf{u}) \chi^\mu.
    \end{equation}
    It remains to show that as long as $c_{\lambda\mu}$ appears in a Frobenius series expansion, it does not depend on the choice of $(k,j)$. 

    By Lemma~\ref{lem:intersection},
    \begin{equation}\label{eq:claim}
        \langle \Sym(\C^{k|j}\otimes V)^G_+ \rangle\big|_{\bm{x}^{(k)}=0} = \langle \Sym(\C^{k|j}\otimes V)^G_+ \rangle \cap \Sym(\C^{k-1|j}\otimes V) = \langle \Sym(\C^{k-1|j}\otimes V)^G_+ \rangle,
    \end{equation}
    where the first equality holds because the ideal is $q_k$-homogeneous.
    It follows that $R_G^{(k-1,j)} = R_G^{(k,j)}|_{\bm{x}^{(k)}=0}$.
    Since only the variables in $\bm{x}^{(k)}$ contribute to the $q_k$-degree,
    \begin{equation}\label{eq:q_k=0}
        \Char(R_G^{(k-1,j)};q_1,\ldots,q_{k-1};\mathbf{u}) = \Char(R_G^{(k,j)};q_1,\ldots,q_{k-1},q_k;\mathbf{u})|_{q_k = 0}.
    \end{equation}
    
    On the one hand, suppose we are given equation~\eqref{eq:k-j-proof}; then by equation~\eqref{eq:q_k=0}, 
    \begin{align}\label{eq:proof-first}
    \begin{aligned}
        \Char(R_G^{(k-1,j)};q_1,\ldots,q_{k-1};\mathbf{u}) 
        &= \sum_{\lambda \in P(k,j,n)} \sum_{\chi^\mu \in \IrrChar(G)} c_{\lambda\mu} s_\lambda(q_1,\ldots,q_{k-1},q_k/\mathbf{u})|_{q_k =0}\cdot \chi^\mu\\
        &= \sum_{\lambda \in P(k-1,j,n)} \sum_{\chi^\mu \in \IrrChar(G)} c_{\lambda\mu} s_\lambda(q_1,\ldots,q_{k-1},0/\mathbf{u}) \chi^\mu,
    \end{aligned}
    \end{align}
    where the last equality follows because when $q_k=0$, the restriction on which super Schur functions are nonzero changes from $\lambda_{k+1} \leq j$ to $\lambda_{k} \leq j$.
    
    On the other hand, we can compute directly that for some nonnegative integer coefficients $\tilde{c}_{\lambda\mu}$,
    \begin{align}\label{eq:proof-second}
    \begin{aligned}
        \Char(R_G^{(k-1,j)};q_1,\ldots,q_{k-1};\mathbf{u}) &= \sum_{\lambda \in P(k-1,j,n)} \sum_{\chi^\mu \in \IrrChar(G)} \tilde{c}_{\lambda\mu} s_\lambda(q_1,\ldots,q_{k-1}/\mathbf{u}) \chi^\mu\\
        &= \sum_{\lambda \in P(k-1,j,n)} \sum_{\chi^\mu \in \IrrChar(G)} \tilde{c}_{\lambda\mu} s_\lambda(q_1,\ldots,q_{k-1},0/\mathbf{u}) \chi^\mu,
    \end{aligned}
    \end{align}
    where we applied restriction (equation~\eqref{eq:restriction1}). 
    Thus we conclude that $\tilde{c}_{\lambda\mu} = c_{\lambda\mu}$ for all $\lambda \in P(k-1,j,n)$, i.e., for those $\lambda$ which appear in both expansions (equations~\eqref{eq:proof-first} and \eqref{eq:proof-second}). 
    Checking the character series for $R_G^{(k,j-1)}$ is analogous, and we conclude that the $c_{\lambda\mu}$ do not depend on $(k,j)$.
\end{proof}

In the case of the symmetric group $\mathfrak{S}_n$ for a fixed $n$, the purely bosonic case completely determines the mixed bosonic-fermionic case, as long as the number of sets of bosonic variables is large enough (at least $n-1$).
The purely fermionic case completely determines the mixed bosonic-fermionic case, as long as the number of sets of fermionic variables is large enough (at least $\binom{n}{2}$).

\begin{corollary}\label{cor:diagonal-supersymmetry-threshold} 
Fix $n \geq 1$. 
The multigraded Frobenius series of $R_n^{(k,j)}$ is determined by the multigraded Frobenius series of $R_n^{(K,0)}$ for $K\geq n-1$ or by the multigraded Frobenius series of $R_n^{(0,J)}$ for $J \geq \binom{n}{2}$. 
In particular, the coefficients $c_{\lambda \mu}$ in equation~\eqref{eq:theorem} are determined by
\begin{equation}\label{eq:bosonic}
    \Frob(R_n^{(K,0)}; q_1,\ldots, q_{K}) = \sum_{\lambda \in P(K,0,n)} \sum_{\mu \vdash n}  c_{\lambda\mu} s_\lambda(q_1,\ldots, q_{K}) s_\mu(\mathbf{z}),
\end{equation}
or by
\begin{equation}\label{eq:fermionic}
    \Frob(R_n^{(0,J)}; u_1,\ldots, u_{J}) = \sum_{\lambda \in P(0,J,n)} \sum_{\mu \vdash n} c_{\lambda\mu} s_{\lambda'}(u_1,\ldots, u_{J}) s_\mu(\mathbf{z}).
\end{equation}
\end{corollary}

\begin{proof}
    By Corollary~\ref{cor:main-theorem}, it remains to check that equation~\eqref{eq:bosonic} or equation~\eqref{eq:fermionic} determines the $c_{\lambda \mu}$. 

    For the bosonic case, by \cite{Bergeron2013} the multigraded Frobenius series of $R_n^{(K,0)}$ is coefficient stable once $K \geq n-1$, so all coefficients are determined by equation~\eqref{eq:bosonic}. 

    For the fermionic case, recall from Corollary~\ref{cor:main-theorem} that the $c_{\lambda\mu}$ do not depend on $(k,j)$. 
    Hence it suffices to determine which of them appear in equation~\eqref{eq:fermionic}.
    Since
    \begin{equation}
        s_{\lambda}(0/u_1,\ldots,u_J) = s_{\lambda'}(u_1,\ldots,u_J)
    \end{equation}
    is nonzero if and only if $\ell(\lambda') = \lambda_1 \leq J$, then $\Frob(R_n^{(0,J)})$ determines exactly those $c_{\lambda\mu}$ with $\lambda_1 \leq J$. 
    It therefore suffices to show that $\lambda_1 \leq \binom{n}{2}$ for every $\lambda$ indexing a nonzero $c_{\lambda\mu}$.

    We bound $\lambda_1$ using the purely bosonic coinvariant ring. 
    For one set of bosonic variables, the dimension of the homogeneous component $(R_n^{(1,0)})_{r}$ is $0$ for all $r > \binom{n}{2}$ (this follows from the fact that $\Hilb(R_n^{(1,0)};q) = [n]_q!$ \cite{Artin}). 
    By considering the generators for their defining ideals, the ring $R_n^{(K,0)}$ is a quotient module of $(R_n^{(1,0)})^{\otimes K}$.
    Consequently, $(R_n^{(K,0)})_{r_1,\ldots,r_K} = 0$ whenever some $r_i > \binom{n}{2}$. 

    Now suppose $c_{\lambda \mu} > 0$ and take $K \geq \ell(\lambda)$.
    By Theorem~\ref{thm:new-main-first-part}, $R_n^{(K,0)}$ contains the simple $\GL(K)$-module $U^\lambda_K$, whose highest weight vector has multidegree $(\lambda_1,\ldots,\lambda_K)$, so $(R_n^{(K,0)})_{\lambda_1,\ldots,\lambda_K} \neq 0$. 
    By the vanishing established above, $\lambda_1 \leq \binom{n}{2}$.
    Therefore once $J \geq \binom{n}{2}$, every nonzero $c_{\lambda\mu}$ is determined by
    equation~\eqref{eq:fermionic}.
\end{proof}
Note that the bounds in Corollary~\ref{cor:diagonal-supersymmetry-threshold} are tight (see Remark~\ref{rmk:tightness}).

Let $R_n^{(\infty,0)}$ be the direct limit of the $R_n^{(k,0)}$, and let
$R_n^{(0,\infty)}$ be the direct limit of the $R_n^{(0,j)}$, with transition maps as in Section~\ref{sec:introduction}. 
By taking the limit as $k$ or $j$ goes to infinity, we recover the form, with infinitely many sets of variables, in which F. Bergeron framed this determination \cite{Bergeron2020, BIRS}: the mixed bosonic-fermionic case may be determined from the exclusively bosonic or the exclusively fermionic case.

\begin{corollary}\label{cor:infinite-threshold}
Fix $n \geq 1$. The multigraded Frobenius series of $R_n^{(\infty,\infty)}$ may be calculated from the multigraded Frobenius series of $R_n^{(\infty,0)}$ or $R_n^{(0,\infty)}$. In particular, the coefficients $c_{\lambda \mu}$ in equation~\eqref{eq:theorem} are determined by
\begin{equation}
    \Frob(R_n^{(\infty,0)}; q_1,q_2,\ldots) = \sum_{\lambda, \ \ell(\lambda) \leq n} \sum_{\mu \vdash n}  c_{\lambda\mu} s_\lambda(q_1,q_2,\ldots) s_\mu(\mathbf{z}),
\end{equation}
or by
\begin{equation}
    \Frob(R_n^{(0,\infty)}; u_1,u_2,\ldots) = \sum_{\lambda, \ \ell(\lambda) \leq n} \sum_{\mu \vdash n} c_{\lambda\mu} s_{\lambda'}(u_1,u_2,\ldots) s_\mu(\mathbf{z}).
\end{equation}
\end{corollary}

\section{Universal series for the symmetric group}\label{sec:discussion}

Currently, combinatorial or algebraic formulas have been proven for the Frobenius series of $R_n^{(k,j)}$ only in the cases of $R_n^{(1,0)}$ \cite{Chevalley}, $R_n^{(2,0)}$ \cite{Haiman2002}, $R_n^{(0,1)}$,\footnote{Since $R_n^{(0,1)} = \C[\theta_1,\ldots,\theta_n]/\langle \theta_1+\cdots+\theta_n\rangle \cong \C[\theta_1,\ldots,\theta_{n-1}]$, this just reduces to an exterior algebra.} $R_n^{(0,2)}$ \cite{KimRhoades2022}, and $R_n^{(1,1)}$ \cite{MuraiRhoadesWilson}. All cases where $k + j \geq 3$ remain open.

A natural question is whether the graded Frobenius series for $R_n^{(1,0)}$ or $R_n^{(0,1)}$ determines the other; similarly whether the bigraded Frobenius series for $R_n^{(2,0)}$, $R_n^{(1,1)}$, or $R_n^{(0,2)}$ determines any of the others.
We will answer these in the negative in Proposition~\ref{prop:negative}.

More generally, fix $k,j \geq 0$ and consider the following.
Suppose we know the following multigraded Frobenius series expansion for fixed $n$:
\begin{equation}\label{eq:expansion}
    \Frob(R_n^{(k,j)}; \mathbf{q};\mathbf{u}) = \sum_{\lambda \in P(k,j,n)}\sum_{\mu \vdash n}  c_{\lambda\mu} s_\lambda(\mathbf{q}/\mathbf{u})s_\mu(\mathbf{z}).
\end{equation}

For any $0 \leq k' \leq k$ and $0 \leq j' \leq j$, because of restriction (equations~\eqref{eq:restriction1} and \eqref{eq:restriction2}), this determines the $c_{\lambda\mu}$ necessary to expand
\begin{equation}\label{eq:expansion-prime}
    \Frob(R_n^{(k',j')}; q_1,\ldots,q_{k'};u_1,\ldots,u_{j'}) = \sum_{\lambda \in P(k',j',n)}\sum_{\mu \vdash n}  c_{\lambda\mu} s_\lambda(q_1,\ldots,q_{k'}/u_1,\ldots,u_{j'})s_\mu(\mathbf{z}).
\end{equation}

Now we can prove Proposition~\ref{prop:universal-series}.

\begin{proof}[Proof of Proposition~\ref{prop:universal-series}]
The above discussion justifies calling
\begin{equation}
    \Frob(R_n^{(\infty,\infty)}; \mathbf{q};\mathbf{u}) =  \sum_{\lambda,\ \ell(\lambda) \leq n} \sum_{\mu \vdash n}c_{\lambda\mu} s_\lambda(\mathbf{q}/\mathbf{u})s_\mu(\mathbf{z})
\end{equation}
the universal series, since this determines all $c_{\lambda\mu}$ which could appear in such an expansion for any $k,j$.
\end{proof}

To show that the multigraded Frobenius series for $(k',j')$, where $k' > k$ or $j' > j$, cannot be determined from the multigraded Frobenius series for $(k,j)$, it suffices to exhibit specific $\hat{\lambda}, \hat{\mu}$ such that:
\begin{itemize}
    \item $c_{\hat{\lambda} \hat{\mu}} > 0$,
    \item $\hat{\lambda}_{k'+1} \leq j'$, and
    \item $\hat{\lambda}_{k+1} > j$.
\end{itemize}

We make the following natural conjecture.

\begin{conjecture}
For sufficiently large $n$, there are nonzero $c_{\lambda\mu}$ in the expansion of the multigraded Frobenius series of $R_n^{(k',j')}$ (equation~\eqref{eq:expansion-prime}) which are not determined by the expansion of the  multigraded Frobenius series of $R_n^{(k,j)}$ (equation~\eqref{eq:expansion}) when $(k,j)$ and $(k',j')$ are incomparable in the componentwise order.
\end{conjecture}

Then we verify some small cases of this conjecture with the following result.

\begin{proposition}\label{prop:negative}
For sufficiently large $n$, there are nonzero $c_{\lambda\mu}$ in the expansion of the multigraded Frobenius series of $R_n^{(k',j')}$ (equation~\eqref{eq:expansion-prime}) which are not determined by the expansion of the multigraded Frobenius series of $R_n^{(k,j)}$ (equation~\eqref{eq:expansion}) when
$(k,j)$ and $(k',j')$ are either permutation of $(1,0)$ and $(0,1)$, or when $(k,j)$ and $(k',j')$ are any two distinct choices of $(2,0)$, $(1,1)$, and $(0,2)$.
\end{proposition}

Before proving Proposition~\ref{prop:negative}, we will find formulas for certain $c_{\lambda\mu}$ to use in the proof.
The following result classifies all $c_{\lambda\mu}$ which we can determine from $\Frob(R_n^{(0,2)};u,v)$.
\begin{proposition}\label{prop:coefficients-KR}
    For $\lambda$ a partition where each part is at most 2, and $\mu \vdash n$, the coefficients $c_{\lambda\mu}$ defined by
    \begin{equation}
        \Frob(R_n^{(\infty,\infty)}; \mathbf{q};\mathbf{u}) =\sum_{\lambda \in P(\infty,\infty,n)} \sum_{\mu \vdash n}  c_{\lambda\mu} s_\lambda(\mathbf{q}/\mathbf{u})s_\mu(\mathbf{z})
    \end{equation}
    are 1 if:
    \begin{enumerate}[(i)]
        \item $\mu = (n)$ and $\lambda = \varnothing$;
        \item $\mu = (1^n)$ and $\lambda = (1^{n-1})$;
        \item $\mu = (n-i,1^i)$ and $\lambda = (1^i)$ or $(2,1^{i-1})$ for $i \in \{1,\ldots,n-2\}$;
        \item $\mu = (\mu_1,\mu_1, 2^\ell, 1^{n-2\ell-2\mu_1})$ where $\mu_1 \geq 2$ and $\lambda = (2^{\ell+\mu_1}, 1^{n-2\ell-2\mu_1-1})$, $(2^{\ell+\mu_1-1}, 1^{n-2\ell-2\mu_1+1})$, or $(2^{\ell+\mu_1-1}, 1^{n-2\ell-2\mu_1})$;
        \item $\mu = (\mu_1,\mu_2, 2^\ell, 1^{n-2\ell-\mu_1-\mu_2})$ where $\mu_1 > \mu_2 \geq 2$ and $\lambda = (2^{\ell+\mu_2}, 1^{n-2\ell-\mu_1-\mu_2})$, $(2^{\ell+\mu_2-1}, 1^{n-2\ell-\mu_1-\mu_2+1})$, $(2^{\ell+\mu_2}, 1^{n-2\ell-\mu_1-\mu_2-1})$, or $(2^{\ell+\mu_2-1}, 1^{n-2\ell-\mu_1-\mu_2})$,
    \end{enumerate}
    and $0$ otherwise.
\end{proposition}

\begin{proof}
Consider
\begin{equation}
        \Frob(R_n^{(0,2)}; u,v) = \sum_{\lambda \in P(0,2,n)} \sum_{\mu \vdash n}  c_{\lambda\mu} s_\lambda(0/u,v)s_\mu(\mathbf{z}).
    \end{equation}
    It follows from equation~\eqref{eq:superschur} that $s_\lambda(0/u,v) = s_{\lambda'}(u,v)$.
    From \cite[Proposition 6.2]{KimRhoades2022}, the coefficient of $s_{(n)}(\mathbf{z})$ is $1 = s_\varnothing(u,v) = s_\varnothing(0/u,v)$, the coefficient of $s_{(1^n)}(\mathbf{z})$ is $[n]_{u,v} = s_{(n-1)}(u,v) = s_{(1^{n-1})}(0/u,v)$, and for any $i \in \{1, \ldots, n-2\}$, the coefficient of $s_{(n-i,1^i)}(\mathbf{z})$ is $[i+1]_{u,v} + uv[i]_{u,v} = s_{(i)}(u,v) +s_{(i,1)}(u,v) = s_{(1^i)}(0/u,v) +s_{(2,1^{i-1})}(0/u,v)$.
    Similarly, from \cite[Theorem 5.2]{Lentfer2025} we determine $(iv)$ and $(v)$.

    Moreover, since \cite[Proposition 6.2]{KimRhoades2022} and \cite[Theorem 5.2]{Lentfer2025} together give the complete expansion of $\Frob(R_n^{(0,2)};u,v)$, all other coefficients $c_{\lambda\mu}$ where each part of $\lambda$ is at most $2$ are zero.
\end{proof}

Swanson--Wallach \cite[Corollary 5.8]{SwansonWallach1} found the bigraded multiplicity of the sign character in $R_n^{(1,1)}$, which can be expressed in the form (see \cite[equation~(3.14)]{Bergeron2020}):
\begin{equation}\label{eq:Swanson--Wallach}
    \langle \Frob(R_n^{(1,1)};q;u), s_{(1^n)}(\mathbf{z})\rangle= \sum_{d = 0}^{n-1} q^{\binom{n-d}{2}}\qbinom{n-1}{d}_q u^d.
\end{equation}
We wish to write this in terms of super Schur functions $s_\lambda(q/u)$.

\begin{lemma}\label{lem:1-1-hook}
    If $\lambda$ is $(a,1^b)$ for some $a \geq 1$ and $b \geq 0$, then $s_\lambda(q/u) = q^au^b + q^{a-1}u^{b+1}$. 
    If $\lambda$ is $\varnothing$, then $s_\lambda(q/u) = 1$.
    If $\lambda$ is not contained in a hook shape, then $s_\lambda(q/u) = 0$.
\end{lemma}

\begin{proof}
    Direct calculation.
\end{proof}

For a polynomial $p \in \C[q]$, let $\langle q^i \rangle p$ denote the coefficient of $q^i$ in $p$.
Define 
    \begin{equation}
        g_{i,d,n}:= \langle q^i \rangle \qbinom{n-2}{d}_q = |\{ \lambda \vdash i : \lambda \subseteq (n-2-d)^d\}|,
    \end{equation}
    where $(n-2-d)^d$ is the rectangular shape of width $n-2-d$ and height $d$.

\begin{proposition}
For all $n\geq 1$,
    \begin{equation}
        \langle \Frob(R_n^{(1,1)};q;u), s_{(1^n)}(\mathbf{z})\rangle = \sum_{d=0}^{n-1} \sum_{i \geq 0} g_{i,d,n}\, s_{(\binom{n-d}{2}+i,1^d)}(q/u).
    \end{equation}
\end{proposition}

\begin{proof}
    For $n=1$, both sides equal $1$. Assume $n \geq 2$.
    By Lemma~\ref{lem:1-1-hook}, we write
    \begin{align}
    \begin{aligned}[c]
        \sum_{d=0}^{n-1} \sum_{i \geq 0} g_{i,d,n}\, s_{(\binom{n-d}{2}+i,1^d)}(q/u) &= \sum_{d=0}^{n-1} \sum_{i \geq 0} \left( q^{\binom{n-d}{2}+i} u^d + q^{\binom{n-d}{2}+i-1}u^{d+1} \right)\langle q^i \rangle \qbinom{n-2}{d}_q\\
        &= \sum_{d=0}^{n-1} u^d \left( \sum_{i \geq 0}  q^{\binom{n-d}{2}+i}\langle q^i \rangle \qbinom{n-2}{d}_q +\sum_{i \geq 0} q^{\binom{n-d+1}{2}+i-1} \langle q^i \rangle \qbinom{n-2}{d-1}_q \right)\\
        &= \sum_{d=0}^{n-1} u^d q^{\binom{n-d}{2}}\left( \qbinom{n-2}{d}_q +q^{n-d-1} \qbinom{n-2}{d-1}_q \right)\\
        &= \sum_{d=0}^{n-1} u^d q^{\binom{n-d}{2}}\qbinom{n-1}{d}_q,
    \end{aligned}
    \end{align}
    where the last line is Pascal's identity. Then apply equation~\eqref{eq:Swanson--Wallach}.
\end{proof}

As a consequence of this proposition and Proposition~\ref{prop:universal-series}, we deduce the following.

\begin{corollary}\label{cor:coeff-1-1-sign}
    For $\lambda$ a partition contained in a hook shape, the coefficients $c_{\lambda, (1^n)}$ defined by
    \begin{equation}
        \Frob(R_n^{(\infty,\infty)}; \mathbf{q};\mathbf{u}) = \sum_{\lambda \in P(\infty,\infty,n)} \sum_{\mu \vdash n} c_{\lambda\mu} s_\lambda(\mathbf{q}/\mathbf{u})s_\mu(\mathbf{z})
    \end{equation}
    are given by:
    \begin{equation}
        c_{\lambda, (1^n)} = \begin{cases}
            g_{i,d,n} & \text{ if there exist }i\geq 0,\, n-1 \geq d \geq 0 \text{ such that } \lambda = (\binom{n-d}{2} + i, 1^d),\\
            0 & \text{ otherwise.}
        \end{cases}
    \end{equation}
\end{corollary}

\begin{remark}\label{rmk:tightness}
    Corollary~\ref{cor:coeff-1-1-sign} tells us that $c_{(\binom{n}{2}),(1^n)}=1$. 
    Consider $R_n^{(0,J)}$. For $J < \binom{n}{2}$, we have
    \begin{equation}
        s_{(\binom{n}{2})}(0/u_1,\ldots,u_J)= s_{(1^{\binom{n}{2}})}(u_1,\ldots,u_J) = 0.
    \end{equation}
    Thus the bound of $J \geq \binom{n}{2}$ is tight in Corollary~\ref{cor:diagonal-supersymmetry-threshold}.

    Proposition~\ref{prop:coefficients-KR}$(ii)$ tells us that $c_{(1^{n-1}), (1^n)} = 1$.
    Consider $R_n^{(K,0)}$. For $K < n-1$, we have
    \begin{equation}
        s_{(1^{n-1})}(q_1,\ldots,q_K)= 0.
    \end{equation}
    Thus the bound of $K \geq n-1$ is tight in Corollary~\ref{cor:diagonal-supersymmetry-threshold}.
\end{remark}

Now we are ready to prove Proposition~\ref{prop:negative}.

\begin{proof}[Proof of Proposition~\ref{prop:negative}]\
    \begin{enumerate}
        \item $R_n^{(1,0)} \nrightarrow  R_n^{(0,1)}$:
        By Proposition~\ref{prop:coefficients-KR}$(ii)$, $c_{\hat{\lambda}\hat{\mu}} =c_{(1^{n-1}), (1^n)} = 1$.
        For all $n \geq 3$, $\hat{\lambda}=(1^{n-1})$ satisfies $\hat{\lambda}_{1}\leq 1$ and $\hat{\lambda}_{2} > 0$.
        \item $R_n^{(0,1)} \nrightarrow  R_n^{(1,0)}$:
        By Corollary~\ref{cor:coeff-1-1-sign}, $c_{\hat{\lambda}\hat{\mu}} = c_{(\binom{n}{2}),(1^n)}=1$.
        For all $n \geq 3$, $\hat{\lambda} = (\binom{n}{2})$ satisfies $\hat{\lambda}_{2}\leq 0$ and $\hat{\lambda}_{1}> 1$.
        \item $R_n^{(2,0)} \nrightarrow  R_n^{(0,2)}$:
        By Proposition~\ref{prop:coefficients-KR}$(ii)$, $c_{\hat{\lambda}\hat{\mu}} =c_{(1^{n-1}), (1^n)} = 1$.
        For all $n \geq 4$, $\hat{\lambda}=(1^{n-1})$ satisfies $\hat{\lambda}_{1}\leq 2$ and $\hat{\lambda}_{3} > 0$.
        \item $R_n^{(0,2)} \nrightarrow  R_n^{(2,0)}$: 
        By Corollary~\ref{cor:coeff-1-1-sign}, $c_{\hat{\lambda}\hat{\mu}} = c_{(\binom{n}{2}),(1^n)}=1$.
        For all $n \geq 3$, $\hat{\lambda} = (\binom{n}{2})$ satisfies $\hat{\lambda}_{3}\leq 0$ and $\hat{\lambda}_{1} > 2$.
        \item $R_n^{(2,0)} \nrightarrow  R_n^{(1,1)}$:
        By Corollary~\ref{cor:coeff-1-1-sign}, $c_{\hat{\lambda}\hat{\mu}} = c_{(\binom{n-2}{2},1,1),(1^n)} = 1$.
        For all $n \geq 4$, $\hat{\lambda} = (\binom{n-2}{2},1,1)$ satisfies $\hat{\lambda}_2 \leq 1$ and $\hat{\lambda}_{3} > 0$.
        \item $R_n^{(0,2)} \nrightarrow  R_n^{(1,1)}$: 
        By Corollary~\ref{cor:coeff-1-1-sign}, $c_{\hat{\lambda}\hat{\mu}} = c_{(\binom{n}{2}),(1^n)}=1$.
        For all $n \geq 3$, $\hat{\lambda} = (\binom{n}{2})$ satisfies $\hat{\lambda}_2 \leq 1$ and $\hat{\lambda}_1 > 2$.
        \item $R_n^{(1,1)} \nrightarrow  R_n^{(0,2)}$:
        By Proposition~\ref{prop:coefficients-KR}$(iv)$, $c_{\hat{\lambda}\hat{\mu}} = c_{(2,2,1^{n-5}),(2,2,1^{n-4})} =1$.
        For all $n\geq 5$, $\hat{\lambda} =(2,2,1^{n-5})$ satisfies $\hat{\lambda}_1 \leq 2$ and $\hat{\lambda}_2 > 1$.
        \item $R_n^{(1,1)} \nrightarrow  R_n^{(2,0)}$: 
        Lee, Li, and Loehr \cite[Appendix]{LeeLiLoehr} determine the $\mathfrak{sl}(2)$-string structure for the homogeneous components of the $q,t$-Catalan number $C_n(q,t) = \langle \Frob(R_n^{(2,0)};q,t),s_{(1^n)}(\mathbf{z})\rangle$ of total degree greater than or equal to $ \binom{n}{2} - 9$. 
        In particular, the homogeneous component of total degree $\binom{n}{2}-2$ is shown to be $s_{(\binom{n}{2}-3,1)}(q,t) + s_{(\binom{n}{2}-4,2)}(q,t)$ for large enough $n$.
        This implies that $c_{\hat{\lambda}\hat{\mu}} = c_{(\binom{n}{2}-4,2),(1^n)} = 1$.
        For large enough $n$, $\hat{\lambda}= (\binom{n}{2}-4,2)$ satisfies $\hat{\lambda}_3 \leq 0$ and $\hat{\lambda}_2 > 1$. 
    \end{enumerate}
\end{proof}

\section{Cancellation and the Hilbert series for the symmetric group}\label{sec:hilbert}

In this section, we discuss further applications of diagonal supersymmetry for the symmetric group. 
In particular, we consider the Hilbert series, and obtain a diagonal supersymmetry result for Hilbert series in Proposition~\ref{prop:hilbert-version}. 
Super Schur functions satisfy a cancellation property (equation~\eqref{eq:cancellation}) which has no analogue for classical Schur functions. 
We are able to leverage cancellation in Proposition~\ref{prop:general-cancellation} and then explain how several conjectures or results on symmetric functions in the literature relate to conjectures or results on the Frobenius and Hilbert series of coinvariant rings.

The multigraded Hilbert series is defined by 
\begin{equation} \Hilb(R_n^{(k,j)}; \mathbf{q}; \mathbf{u}) := \sum_{r_1,\ldots,r_k, s_1, \ldots, s_j \geq 0}\dim \left((R_n^{(k,j)})_{r_1,\ldots,r_k, s_1, \ldots, s_j} \right)q_1^{r_1} \cdots q_k^{r_k}u_1^{s_1} \cdots u_j^{s_j}.\end{equation}
Recall that we can recover the Hilbert series from the Frobenius series via: \begin{equation}\langle \Frob(R_n^{(k,j)}; \mathbf{q}; \mathbf{u}), h_{(1^n)}(\mathbf{z}) \rangle = \Hilb(R_n^{(k,j)}; \mathbf{q}; \mathbf{u}).\end{equation}

We give a Hilbert series version of Corollary~\ref{cor:main-theorem}.

\begin{proposition}\label{prop:hilbert-version}
    Fix a positive integer $n$. For partitions $\lambda$ with $\ell(\lambda) \leq n$, there exist nonnegative integer coefficients $c_\lambda(n)$ such that for any $(k,j)$, the multigraded Hilbert series of $R_n^{(k,j)}$ is
\begin{equation}
    \Hilb(R_n^{(k,j)}; \mathbf{q};\mathbf{u}) = \sum_{\lambda \in P(k,j,n)} c_{\lambda}(n) s_\lambda(\mathbf{q}/\mathbf{u}).
\end{equation}
\end{proposition}

\begin{proof}
    By Corollary~\ref{cor:main-theorem},
    \begin{align}
    \begin{aligned}
        \Hilb(R_n^{(k,j)}; \mathbf{q}; \mathbf{u}) &= \langle \Frob(R_n^{(k,j)}; \mathbf{q}; \mathbf{u}), h_{(1^n)}(\mathbf{z}) \rangle\\
        &= \left\langle \sum_{\lambda \in P(k,j,n)}\sum_{\mu \vdash n}  c_{\lambda\mu} s_\lambda(\mathbf{q}/\mathbf{u})s_\mu(\mathbf{z}), h_{(1^n)}(\mathbf{z}) \right\rangle\\
        &= \sum_{\lambda \in P(k,j,n)} s_\lambda(\mathbf{q}/\mathbf{u}) \sum_{\mu \vdash n} c_{\lambda\mu} \langle s_\mu(\mathbf{z}), h_{(1^n)}(\mathbf{z}) \rangle
    \end{aligned}
    \end{align}
    and the formula follows upon setting $c_\lambda(n) := \sum_{\mu \vdash n} c_{\lambda\mu} \langle s_\mu(\mathbf{z}), h_{(1^n)}(\mathbf{z})\rangle$. 
    Since $\langle s_\mu(\mathbf{z}), h_{(1^n)}(\mathbf{z})\rangle$ and $c_{\lambda\mu}$ are both nonnegative integers and independent of $(k,j)$, so are the coefficients $c_\lambda(n)$. 
\end{proof}

Now we show a general result on Frobenius and Hilbert series involving some specific evaluations.

\begin{proposition}\label{prop:general-cancellation}
    For any $n\geq 1$ and for any $0 \leq m \leq \min(k,j)$,
    \begin{enumerate}[(i)]
        \item     $\Frob(R_n^{(k,j)}; \mathbf{q}; \mathbf{u})|_{q_k = -u_j, \ldots, q_{k-m+1} = -u_{j-m+1}} = \Frob(R_n^{(k-m,j-m)};q_1,\ldots,q_{k-m};u_1,\ldots,u_{j-m}),$
    \item   $\Hilb(R_n^{(k,j)}; \mathbf{q}; \mathbf{u})|_{q_k = -u_j, \ldots, q_{k-m+1} = -u_{j-m+1}} = \Hilb(R_n^{(k-m,j-m)};q_1,\ldots,q_{k-m};u_1,\ldots,u_{j-m}).$
    \end{enumerate}
    More generally, the result holds for any $m$ pairs $(q_{i_1}, u_{h_1}),\ldots, (q_{i_m}, u_{h_m})$, with the $\{i_1,\ldots, i_m\}$ pairwise distinct and the $\{h_1,\ldots, h_m\}$ pairwise distinct. That is,
\begin{align*}
            \Frob&(R_n^{(k,j)}; \mathbf{q}; \mathbf{u})|_{q_{i_a} = -u_{h_a} \text{ for all } a \in \{1,\ldots,m\}}\\
            &= \Frob(R_n^{(k-m,j-m)};q_1,\ldots, \widehat{q_{i_1}}, \ldots,\widehat{q_{i_m}}, \ldots,  q_{k};u_1,\ldots, \widehat{u_{h_1}}, \ldots,\widehat{u_{h_m}}, \ldots,u_{j}),
        \end{align*}
and \begin{align*}
            \Hilb&(R_n^{(k,j)}; \mathbf{q}; \mathbf{u})|_{q_{i_a} = -u_{h_a} \text{ for all } a \in \{1,\ldots,m\}}\\
            &= \Hilb(R_n^{(k-m,j-m)};q_1,\ldots, \widehat{q_{i_1}}, \ldots,\widehat{q_{i_m}}, \ldots,  q_{k};u_1,\ldots, \widehat{u_{h_1}}, \ldots,\widehat{u_{h_m}}, \ldots,u_{j}).
    \end{align*}
\end{proposition}

\begin{proof}
    The first claim follows immediately from Corollary~\ref{cor:main-theorem} and $m$ applications of cancellation (equation~\eqref{eq:cancellation}). The second claim follows from the first after applying $\langle -, h_{(1^n)}(\mathbf{z})\rangle$.
    The same proof holds in the more general case since the multigraded Frobenius series is symmetric in $q_1,\ldots,q_k$ and in $u_1,\ldots,u_j$.
\end{proof}

We now discuss some connections with other results and conjectures in the literature.
Sagan--Swanson conjectured \cite{SaganSwanson2024} and Rhoades--Wilson proved \cite{RhoadesWilson2023} that
\begin{equation}\label{eq:Rhoades-Wilson}
    \Hilb(R_n^{(1,1)};q;u) = \sum_{d=0}^n [d]_q! \Stir_q(n,d) u^{n-d},
\end{equation}
where the $q$-Stirling number $\Stir_q(n,d)$ is defined by the recurrence relation $\Stir_q(n,d) = [d]_q \Stir_q(n-1,d) + \Stir_q(n-1,d-1)$ with initial conditions $\Stir_q(0,d) = \delta_{d,0}$.

We give a short proof of the following result, originally proven by Sagan--Swanson with a sign-reversing involution. 

\begin{corollary}[\!\!{\cite[Theorem 5.5]{SaganSwanson2024}}]
    For $n\geq 0$, we have
    \begin{equation}
        \sum_{d=0}^n[d]_q!\Stir_q(n,d) (-q)^{n-d} =1.
    \end{equation}
\end{corollary}

\begin{proof}
    By Proposition~\ref{prop:general-cancellation}, $\Hilb(R_n^{(1,1)};q;u)|_{q=-u} = \Hilb(R_n^{(0,0)}) = 1$. The result follows by equation~\eqref{eq:Rhoades-Wilson}.
\end{proof}

We briefly recall the Delta and nabla operators. The modified Macdonald polynomials $\tilde{H}_\lambda$ form a basis for $\Lambda_{\Q(q,t)}$, the ring of symmetric functions with coefficients in $\Q(q,t)$ (see \cite{HaimanCDM} for example). 
Let 
\begin{equation}
    B_\mu = \sum_{c \in \mu} q^{a'_\mu(c)} t^{l'_\mu(c)} \text{ and } 
    T_\mu = \prod_{c \in \mu} q^{a'_\mu(c)} t^{l'_\mu(c)},
\end{equation}
where the sum or product is taken over all boxes $c$ in the Ferrers diagram of the partition $\mu$. The co-arm statistic $a'_\mu(c)$ is the number of boxes strictly to the left of $c$ in $\mu$ and the co-leg statistic $l'_\mu(c)$ is the number of boxes strictly below $c$ in $\mu$ (in the French notation for partitions: parts are left-justified with the largest part on the bottom).

Nabla $\nabla$ \cite{BergeronGarsia} is defined to be an eigenoperator on the Macdonald basis: 
\begin{equation}
    \nabla \tilde{H}_\mu = T_\mu \tilde{H}_\mu \text{ for all $\mu$}.
\end{equation}

Let $f$ be any symmetric function in $\Lambda_{\Q(q,t)}$. The more general Delta operators $\Delta_f$ and $\Delta_f'$ \cite{HaglundRemmelWilson2018, BergeronGarsiaHaimanTesler} are eigenoperators on the Macdonald basis defined by 
\begin{equation}
    \Delta_f \tilde{H}_\mu = f[B_\mu]\tilde{H}_\mu \text{ and } \Delta_f' \tilde{H}_\mu = f[B_\mu-1]\tilde{H}_\mu \text{ for all $\mu$}.
\end{equation}
Here, $f[\cdot]$ denotes plethystic substitution (see for example \cite{HaimanCDM}).

Zabrocki conjectured \cite{Zabrocki2019} that
\begin{equation}\label{eq:Zabrocki-conjecture}
    \Frob(R_n^{(2,1)};q,t;u) = \sum_{d=0}^{n-1} u^d \Delta'_{e_{n-d-1}}e_n(\mathbf{z}).
\end{equation}
D'Adderio, Iraci, and Vanden Wyngaerd \cite[Theorem 4.11]{DIVW2021-Delta-square} proved that
\begin{equation}\label{eq:DIVW-identity}
    \sum_{d=0}^{n-1} (-q)^d \Delta'_{e_{n-d-1}}e_n(\mathbf{z}) = \nabla e_n(\mathbf{z})|_{q=0}.
\end{equation}
Haiman \cite{Haiman2002} showed that $\Frob(R_n^{(2,0)};q,t) = \nabla e_n(\mathbf{z}).$
Combining this with Proposition~\ref{prop:general-cancellation}, we have
\begin{equation}\label{eq:frob(1,0)}
    \Frob(R_n^{(2,1)};q,t;u)|_{q=-u} = \Frob(R_n^{(1,0)};t) = \nabla e_n(\mathbf{z})|_{q=0}.
\end{equation}
So if equation~\eqref{eq:Zabrocki-conjecture} is true, then equation~\eqref{eq:DIVW-identity} follows from  equation~\eqref{eq:frob(1,0)}.

Moving to the level of Hilbert series, Zabrocki's conjecture implies that
\begin{equation}\label{eq:Zabrocki-conjecture-hilbert}
    \Hilb(R_n^{(2,1)};q,t;u) = \sum_{d=0}^{n-1} \langle u^d \Delta'_{e_{n-d-1}}e_n(\mathbf{z}), h_{(1^n)}(\mathbf{z}) \rangle.
\end{equation}
From equation~\eqref{eq:DIVW-identity}, Corteel, Josuat-Verg\`es, and Vanden Wyngaerd \cite[equation~(6)]{CorteelJVVW} derived
\begin{equation}\label{eq:CJVVW(6)}
    \sum_{d=0}^{n-1} \langle u^d \Delta'_{e_{n-d-1}}e_n(\mathbf{z}), h_{(1^n)}(\mathbf{z}) \rangle\bigg|_{q=-u=-1} = [n]_t!.
\end{equation}
From Proposition~\ref{prop:general-cancellation}, we have
\begin{equation}\label{eq:hilb(1,0)}
    \Hilb(R_n^{(2,1)};q,t;u)|_{q=-u} = \Hilb(R_n^{(1,0)};t) = [n]_t!,
\end{equation}
where the second equality is due to Artin \cite{Artin}. 
So if equation~\eqref{eq:Zabrocki-conjecture-hilbert} is true, then equation~\eqref{eq:CJVVW(6)} follows from  equation~\eqref{eq:hilb(1,0)}.

D'Adderio, Iraci, and Vanden Wyngaerd \cite{DadderioIraciVandenWyngaerd2021} introduced the Theta operators $\Theta_f: \Lambda_{\Q(q,t)} \to \Lambda_{\Q(q,t)}$, in connection with their work on the coinvariant ring $R_n^{(2,2)}$, with two sets of commuting variables and two sets of anticommuting variables.
Like $\Delta_f$, the Theta operators depend on a choice of a symmetric function $f \in \Lambda_{\Q(q,t)}$. 
Unlike $\Delta_f$, they are not eigenoperators for the Macdonald polynomials $\tilde{H}_\lambda$. 
Corteel, Josuat-Verg\`es, and Vanden Wyngaerd \cite[Section 6.4]{CorteelJVVW} conjectured that
\begin{equation}\label{eq:CJVVW-conj}
    \langle\Theta_{e_d}\Theta_{e_\ell} \nabla e_{n-d-\ell}(\mathbf{z}), h_{(1^n)}(\mathbf{z})\rangle|_{q=-1}
\end{equation}
is $t$-positive. D'Adderio, Iraci, and Vanden Wyngaerd \cite{DadderioIraciVandenWyngaerd2021} conjectured that
\begin{equation}
\label{eq:thetaconj} \Frob(R_n^{(2,2)}; q,t;u,v) = \sum_{\substack{d,\ell \geq 0,\\ d + \ell < n}} u^d v^\ell\Theta_{e_d}\Theta_{e_\ell}\nabla e_{n-d-\ell}(\mathbf{z}),
\end{equation}
which implies that
\begin{equation}
\label{eq:thetaconj-hilb} \Hilb(R_n^{(2,2)}; q,t;u,v) = \sum_{\substack{d,\ell \geq 0,\\ d + \ell < n}} u^d v^\ell \langle \Theta_{e_d}\Theta_{e_\ell}\nabla e_{n-d-\ell}(\mathbf{z}), h_{(1^n)}(\mathbf{z})\rangle.
\end{equation}
From Proposition~\ref{prop:general-cancellation}, we have
\begin{equation}\label{eq:hilb(1,1)}
    \Hilb(R_n^{(2,2)};q,t;u,v)|_{q=-v=-1} = \Hilb(R_n^{(1,1)};t;u),
\end{equation}
which is clearly $t$-positive. Thus if equation~\eqref{eq:thetaconj-hilb} is true, then
\begin{equation}
    \sum_{\substack{d,\ell \geq 0,\\ d + \ell < n}} u^d \langle\Theta_{e_d}\Theta_{e_\ell} \nabla e_{n-d-\ell}(\mathbf{z}), h_{(1^n)}(\mathbf{z})\rangle|_{q=-1}
\end{equation}
is $t$-positive. By considering the components graded by $u$-degree, this would imply that, for each $0 \leq d \leq n-1$, the following is $t$-positive:
\begin{equation}
    \sum_{0 \leq \ell < n-d} \langle\Theta_{e_d}\Theta_{e_\ell} \nabla e_{n-d-\ell}(\mathbf{z}), h_{(1^n)}(\mathbf{z})\rangle|_{q=-1},
\end{equation}
which is a slightly weaker result than equation~\eqref{eq:CJVVW-conj} being $t$-positive.

\section*{Acknowledgements}

The author would like to thank Fran\c{c}ois Bergeron, Sylvie Corteel, Nicolle Gonz\'alez, Mark Haiman, Eric Jankowski, Pablo Ocal, Brendon Rhoades, Chris Ryba, Vera Serganova, and Josh Swanson for helpful conversations or feedback.
Thank you to the referees for their valuable feedback.
The author was partially supported by the National Science Foundation Graduate Research Fellowship DGE-2146752.

\bibliographystyle{amsplain}
\bibliography{biblio}

\providecommand{\bysame}{\leavevmode\hbox to3em{\hrulefill}\thinspace}
\providecommand{\MR}{\relax\ifhmode\unskip\space\fi MR }
\providecommand{\MRhref}[2]{%
  \href{http://www.ams.org/mathscinet-getitem?mr=#1}{#2}
}
\providecommand{\href}[2]{#2}
\begin{thebibliography}{10}

\bibitem{Angarone2024}
Robert Angarone, Patricia Commins, Trevor Karn, Satoshi Murai, and Brendon
  Rhoades, \emph{Superspace coinvariants and hyperplane arrangements}, Adv.
  Math. \textbf{467} (2025), 36.

\bibitem{Artin}
Emil Artin, \emph{Galois theory. 2nd ed. {Edited} and supplemented with a
  section on applications by {Arthur} {N}. {Milgram}}, Notre Dame Math. Lect.,
  vol.~2, Univ. of Notre Dame Press, Notre Dame, IN, 1944 (English).

\bibitem{BereleRegev}
A.~Berele and A.~Regev, \emph{Hook {Young} diagrams with applications to
  combinatorics and to representations of {Lie} superalgebras}, Adv. Math.
  \textbf{64} (1987), 118--175 (English).

\bibitem{BHIR}
Fran\c{c}ois Bergeron, Jim Haglund, Alessandro Iraci, and Marino Romero,
  \emph{Bosonic-fermionic diagonal coinvariants and {Theta} operators},
  preprint (2023),
  \url{https://www2.math.upenn.edu/~jhaglund/preprints/BF2.pdf}.

\bibitem{BIRS}
Fran\c{c}ois Bergeron, Jim Haglund, and Mike Zabrocki, \emph{Final report from
  {BIRS} workshop 19w5131: Representation theory connections to $(q,
  t)$-combinatorics}, 2019,
  \url{https://www.birs.ca/workshops/2019/19w5131/report19w5131.pdf}.

\bibitem{Bergeron2013}
Fran{\c{c}}ois Bergeron, \emph{Multivariate diagonal coinvariant spaces for
  complex reflection groups}, Adv. Math. \textbf{239} (2013), 97--108
  (English).

\bibitem{Bergeron2020}
\bysame, \emph{The bosonic-fermionic diagonal coinvariant modules conjecture},
  preprint (2020), \url{https://arxiv.org/abs/2005.00924}.

\bibitem{BergeronOPAC}
\bysame, \emph{$({GL}_k\times {S}_n)$-modules of multivariate diagonal
  harmonics}, Open Problems in Algebraic Combinatorics (Christine Berkesch,
  Benjamin Brubaker, Gregg Musiker, Pavlo Pylyavskyy, and Victor Reiner, eds.),
  Proc. Symp. Pure Math., vol. 110, Providence, RI: American Mathematical
  Society (AMS), 2024, pp.~1--22.

\bibitem{BergeronGarsia}
Fran{\c{c}}ois Bergeron and Adriano~M. Garsia, \emph{Science fiction and
  {Macdonald}'s polynomials}, Algebraic methods and \(q\)-special functions,
  Providence,~RI: American Mathematical Society, 1999, pp.~1--52 (English).

\bibitem{BergeronGarsiaHaimanTesler}
Fran{\c{c}}ois Bergeron, Adriano~M. Garsia, Mark Haiman, and Glenn Tesler,
  \emph{Identities and positivity conjectures for some remarkable operators in
  the theory of symmetric functions}, Methods Appl. Anal. \textbf{6} (1999),
  no.~3, 363--420 (English).

\bibitem{BergeronPrevilleRatelle}
Fran{\c{c}}ois Bergeron and Louis-Fran{\c{c}}ois Pr{\'e}ville-Ratelle,
  \emph{Higher trivariate diagonal harmonics via generalized {Tamari} posets},
  J. Comb. \textbf{3} (2012), no.~3, 317--341.

\bibitem{BHMPS-Delta}
Jonah Blasiak, Mark Haiman, Jennifer Morse, Anna Pun, and George~H. Seelinger,
  \emph{A proof of the extended delta conjecture}, Forum Math. Pi \textbf{11}
  (2023), Paper No. e6, 28.

\bibitem{BHMPS-Paths}
\bysame, \emph{A shuffle theorem for paths under any line}, Forum Math. Pi
  \textbf{11} (2023), Paper No. e5, 38.

\bibitem{BourbakiAlgebraIIX}
Nicolas Bourbaki, \emph{Elements of mathematics. {Algebra}. {Chapter} 8.
  {Translated} from the 2nd {French} edition by {Reinie} {Ern{\'e}}}, Cham:
  Springer, 2023 (English).

\bibitem{CarlssonMellit2018}
Erik Carlsson and Anton Mellit, \emph{A proof of the shuffle conjecture}, J.
  Am. Math. Soc. \textbf{31} (2018), no.~3, 661--697.

\bibitem{ChengWang}
Shun-Jen Cheng and Weiqiang Wang, \emph{Howe duality for {Lie} superalgebras},
  Compos. Math. \textbf{128} (2001), no.~1, 55--94 (English).

\bibitem{ChengWangBook}
\bysame, \emph{Dualities and representations of {L}ie superalgebras}, Graduate
  Studies in Mathematics, vol. 144, American Mathematical Society, Providence,
  RI, 2012.

\bibitem{Chevalley}
Claude Chevalley, \emph{Invariants of finite groups generated by reflections},
  American Journal of Mathematics \textbf{77} (1955), no.~4, 778--782.

\bibitem{CorteelJVVW}
Sylvie Corteel, Matthieu Josuat-Verg{\`e}s, and Anna Vanden~Wyngaerd,
  \emph{Combinatorics of the {Delta} conjecture at {{\(q=-1\)}}}, Algebr. Comb.
  \textbf{7} (2024), no.~1, 17--35 (English).

\bibitem{DIVW2021-Delta-square}
Michele D'Adderio, Alessandro Iraci, and Anna Vanden~Wyngaerd, \emph{The
  {D}elta square conjecture}, Int. Math. Res. Not. \textbf{2021} (2021), no.~1,
  38--84 (English).

\bibitem{DadderioIraciVandenWyngaerd2021}
\bysame, \emph{Theta operators, refined {Delta} conjectures, and coinvariants},
  Adv. Math. \textbf{376} (2021), 60, Id/No 107447.

\bibitem{DAdderioMellit}
Michele D’Adderio and Anton Mellit, \emph{A proof of the compositional delta
  conjecture}, Advances in Mathematics \textbf{402} (2022), 108342.

\bibitem{HaglundRemmelWilson2018}
J.~Haglund, J.~B. Remmel, and A.~T. Wilson, \emph{The delta conjecture}, Trans.
  Am. Math. Soc. \textbf{370} (2018), no.~6, 4029--4057.

\bibitem{Haiman1994}
Mark Haiman, \emph{Conjectures on the quotient ring by diagonal invariants}, J.
  Algebr. Comb. \textbf{3} (1994), no.~1, 17--76.

\bibitem{Haiman2002}
\bysame, \emph{Vanishing theorems and character formulas for the {Hilbert}
  scheme of points in the plane}, Invent. Math. \textbf{149} (2002), no.~2,
  371--407.

\bibitem{HaimanCDM}
\bysame, \emph{Combinatorics, symmetric functions, and {Hilbert} schemes},
  Current developments in mathematics, 2002. Proceedings of the joint seminar
  by MIT and Harvard, Cambridge, MA, 2002, Somerville, MA: International Press,
  2003, pp.~39--111 (English).

\bibitem{Howe1989}
Roger Howe, \emph{Remarks on classical invariant theory}, Trans. Amer. Math.
  Soc. \textbf{313} (1989), no.~2, 539--570.

\bibitem{Howe}
\bysame, \emph{Perspectives on invariant theory: {Schur} duality,
  multiplicity-free actions and beyond}, The Schur lectures (1992), Ramat-Gan:
  Bar-Ilan University; Providence, RI: American Mathematical Society
  (Distrib.), 1995, pp.~1--182 (English).

\bibitem{KimRhoades2022}
Jongwon Kim and Brendon Rhoades, \emph{Lefschetz theory for exterior algebras
  and fermionic diagonal coinvariants}, Int. Math. Res. Not. \textbf{2022}
  (2022), no.~4, 2906--2933.

\bibitem{LeeLiLoehr}
Kyungyong Lee, Li~Li, and Nicholas~A. Loehr, \emph{A combinatorial approach to
  the symmetry of \(q,t\)-{C}atalan numbers}, SIAM Journal on Discrete
  Mathematics \textbf{32} (2018), no.~1, 191--232.

\bibitem{Lentfer2025}
John Lentfer, \emph{The sign character of the triagonal fermionic coinvariant
  ring}, Electron. J. Comb. \textbf{33} (2026), no.~2, Paper P2.3, 22.

\bibitem{Macdonald}
Ian~Grant Macdonald, \emph{Symmetric functions and {Hall} polynomials.}, 2nd
  ed. ed., Oxford: Clarendon Press, 1995.

\bibitem{Mellit}
Anton Mellit, \emph{Toric braids and {$(m,n)$}-parking functions}, Duke Math.
  J. \textbf{170} (2021), no.~18, 4123--4169.

\bibitem{MuraiRhoadesWilson}
Satoshi Murai, Brendon Rhoades, and Andy Wilson, \emph{A proof of the fields
  conjectures}, preprint (2025), \url{https://arxiv.org/abs/2505.24027}.

\bibitem{Musson}
Ian~M. Musson, \emph{Lie superalgebras and enveloping algebras}, Grad. Stud.
  Math., vol. 131, Providence, RI: American Mathematical Society (AMS), 2012
  (English).

\bibitem{RhoadesWilson2023}
Brendon Rhoades and Andrew~Timothy Wilson, \emph{The {Hilbert} series of the
  superspace coinvariant ring}, Forum Math. Pi \textbf{12} (2024), 35
  (English), Id/No e16.

\bibitem{SaganSwanson2024}
Bruce~E. Sagan and Joshua~P. Swanson, \emph{{{\(q\)}}-{S}tirling numbers in
  type {{\(B\)}}}, Eur. J. Comb. \textbf{118} (2024), 35, Id/No 103899.

\bibitem{StanleyEC2}
Richard~P. Stanley, \emph{Enumerative combinatorics. {Volume} 2}, Camb. Stud.
  Adv. Math., vol.~62, Cambridge: Cambridge University Press, 1999.

\bibitem{SwansonWallach1}
Joshua~P. Swanson and Nolan~R. Wallach, \emph{{Harmonic differential forms for
  pseudo-reflection groups I. Semi-invariants}}, Journal of Combinatorial
  Theory, Series A \textbf{182} (2021), Paper no. 105474.

\bibitem{SwansonWallach2}
\bysame, \emph{{Harmonic differential forms for pseudo-reflection groups II.
  Bi-degree bounds}}, Combinatorial Theory \textbf{3} (2023), no.~3, Paper no.
  17.

\bibitem{Zabrocki2019}
Mike Zabrocki, \emph{A module for the {D}elta conjecture}, preprint (2019),
  \url{https://arxiv.org/abs/1902.08966}.

\bibitem{Zabrocki2020}
\bysame, \emph{Coinvariants and harmonics}, 2020,
  \url{https://realopacblog.wordpress.com/2020/01/26/coinvariants-and-harmonics/}.

\end{thebibliography}

\end{document}